\documentclass{article}
\usepackage{amsmath, amssymb, amsthm, verbatim, mathabx, breqn, mathtools, color, graphicx, complexity, tikz, fullpage, libertine}
\usepackage[pdftex,colorlinks=true]{hyperref}
\usepackage[utf8]{inputenc}

\definecolor{darkred}{rgb}{.5,0,0}
\definecolor{darkblue}{rgb}{0,0,.5}

\newcommand{\calG}{\mathcal{G}}
\newcommand{\td}{\mathit{td}}
\newcommand{\tw}{\mathit{tw}}
\newcommand{\pw}{\mathit{pw}}

\newtheorem{theorem}{Theorem}
\newtheorem{lemma}{Lemma}
\newtheorem{proposition}{Proposition}
\newtheorem{definition}{Definition}
\newtheorem{corollary}{Corollary}
\newtheorem{observation}{Observation}
\newtheorem{question}{Question}

\title{Diameter estimates for graph associahedra}

\author{Jean Cardinal \and Lionel Pournin \and Mario Valencia-Pabon}

\date{}

\begin{document}
\maketitle

  \begin{abstract}
Graph associahedra are generalized permutohedra arising as special cases of nestohedra and hypergraphic polytopes. The graph associahedron of a graph $G$ encodes the combinatorics of search trees on $G$, defined recursively by a root $r$ together with search trees on each of the connected components of $G-r$. In particular, the skeleton of the graph associahedron is the rotation graph of those search trees. We investigate the diameter of graph associahedra as a function of some graph parameters. We give a tight bound of $\Theta(m)$ on the diameter of trivially perfect graph associahedra on $m$ edges. We consider the maximum diameter of associahedra of graphs on $n$ vertices and of given tree-depth, treewidth, or pathwidth, and give lower and upper bounds as a function of these parameters. We also prove that the maximum diameter of associahedra of graphs of pathwidth two is $\Theta (n\log n)$. Finally, we give the exact diameter of the associahedra of complete split and of unbalanced complete bipartite graphs.
\end{abstract}

\sloppy

\section{Introduction}

The vertices and edges of a polyhedron form a graph whose diameter (often referred to as the diameter of the polyhedron for short) is related to a number of computational problems. For instance, the question of how large the diameter of a polyhedron can be arises naturally from the study of linear programming and the simplex algorithm (see, for instance \cite{Santos2012} and references therein). The case of associahedra \cite{L04,S63,T51}---whose diameter is known exactly~\cite{P14}---is particularly interesting. Indeed, the diameter of these polytopes is related to the worst-case complexity of rebalancing binary search trees~\cite{SleatorTarjanThurston1988}. Here, we consider the same question on {\em graph associahedra}~\cite{CD06}, a large family of {\em generalized permutohedra} in the sense of Postnikov~\cite{P09} that can be built from an underlying graph.
The question has already been studied by Manneville and Pilaud~\cite{MP15}, by Pournin~\cite{P17} in the special case of cyclohedra, and by Cardinal, Langerman and P\'erez-Lantero~\cite{CLP18} in the special case of tree associahedra. Here, we aim at giving tighter bounds on the diameter of graph associahedra, in terms of some structural invariants of the underlying graphs.

\begin{figure}
\begin{center}
\includegraphics[page=1,scale=.6]{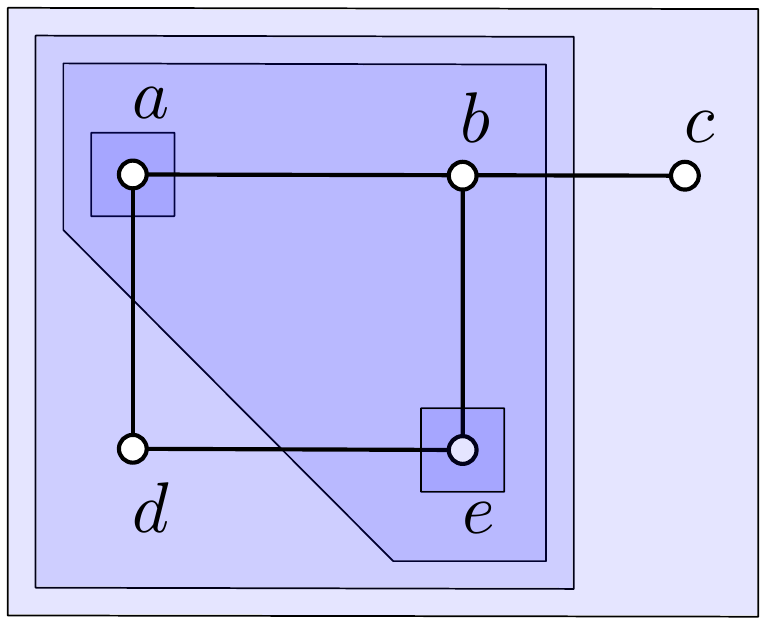}
\includegraphics[page=2,scale=.6]{figures.pdf}
\end{center}
\caption{\label{fig:tubing}An inclusionwise maximal tubing on a graph (left) and the search tree corresponding to it (right). The tubing is the set $\{\{a,b,c,d,e\},\{a,b,d,e\},\{a,b,e\},\{a\},\{e\}\}$. The tubes
$\{a,b,e\}$ and $\{a\}$ are nested; the tubes $\{a\}$ and $\{e\}$ are non-adjacent.}
\end{figure}

\subsection{Graph associahedra}

Graph associahedra have been defined by several authors including Davis, Januszkiewicz, and Scott~\cite{DJS03}, Carr and Devadoss~\cite{CD06}, and Postnikov~\cite{P09}.
We give their definition in terms of so-called {\em tubings} on graphs, following Carr and Devadoss~\cite{CD06}. 
Thereafter, we let $G=(V,E)$ be a simple, connected graph with $|V|=n$ vertices. 
A {\em tube} $S$ in $G$ is a subset $S\subseteq V$ such that the induced subgraph $G[S]$ is connected. 
We say that a pair $(S,S')$ of tubes is {\em nested} when either $S\subset S'$ or $S'\subset S$, and that it is {\em non-adjacent} when $G[S\cup S']$ is disconnected. A {\em tubing} $\mathcal S$ on $G$ is a collection of tubes, every pair of which is either nested or non-adjacent.
An example of (inclusionwise) maximal tubing on a graph is shown on the left of Figure~\ref{fig:tubing}.

The {\em graph associahedron} $\mathcal A (G)$ of $G$ is a convex polytope whose face lattice is isomorphic to the inclusion order of tubings on $G$. 
In particular, the vertices are the maximal tubings on $G$.
It is known that this polytope is always realizable, either by truncation of the permutohedron~\cite{D09}, or as a Minkowski sums of simplices~\cite{P09}. 
The latter is in fact a generalization of Loday's realization of the associahedron~\cite{L04}.

When $G$ is the complete graph on $n$ vertices, no pair of tubes can be non-adjacent. Therefore, the maximal tubings and the vertices of the graph associahedron are in one-to-one correspondence with the permutations of $n$ elements. In that case, the graph associahedron $\mathcal A(G)$ is simply the $(n-1)$-dimensional permutohedron. Another special case of interest is when the graph $G$ is a path on $n$ vertices. In that case, maximal tubings are a Catalan family, and the graph associahedron of $G$ is the $(n-1)$-dimensional classical associahedron. Similarly, the graph associahedra of cycles are the cyclohedra~\cite{P17}, and the graph associahedra of stars are stellohedra~\cite{CD06,FLS09,PW08}. The number of maximal tubings on a graph $G$ is known as the $G$-Catalan number~\cite{P09}.

Just as in the classical associahedron, whose edges correspond to flips in the triangulations of a convex polygon, the edges of a graph associahedron can be interpreted as {\em flips} in tubings: pairs of maximal tubings whose symmetric difference has size two. An example of such a flip is shown on Figure~\ref{fig:rotation}.

\subsection{Search trees}

We are interested in the structure of the skeleton of the graph associahedron $\mathcal A(G)$. For our purpose, it is useful to consider a representation of the vertices of $\mathcal A(G)$ alternative to the inclusionwise maximal tubings. A {\em search tree} $T$ on $G$ is a rooted tree with vertex set $V$ defined recursively as follows: The root of $T$ is a vertex $r\in V$, and $r$ is connected to the root of search trees on each connected component of $G-r$. We will use the standard terminology related to rooted trees, in particular the {\em parent}, {\em child}, {\em ancestor}, and {\em descendant} relations.
A vertex $v$ together with its descendants in a search tree $T$ form a {\em subtree} of $T$ {\em rooted at} $v$. Search trees are in one-to-one correspondence with maximal tubings. Indeed, the tubes are exactly the subsets of the vertices of $T$ contained in a subtree of $T$, as illustrated on the right of Figure~\ref{fig:tubing}. These trees have appeared under various disguises in different contexts. They are called {\em $\mathcal B$-trees} by Postnikov, Reiner, and Williams~\cite{PW08}, and {\em spines} by Manneville and Pilaud~\cite{MP15}. In the context of polymatroids, they are special cases of the partial orders studied by Bixby, Cunningham, and Topkis~\cite{BCT85}. In combinatorial optimization and graph theory, they can be defined in terms of {\em vertex rankings}~\cite{S89,BDJKKMT98}, {\em ordered colorings}~\cite{KMS95}, or as {\em elimination trees}~\cite{P88}.

\subsection{Flips in tubings and rotations in search trees}

We will interpret the edges of $\mathcal{A}(G)$ as {\em rotations} in the search trees on $G$. A rotation in a search tree $T$ involves a pair $a,b$ of vertices, where $a$ is the parent of $b$ in $T$. The set $A$ of vertices in the subtree rooted at $a$ induces a connected subgraph of $G$, and forms a tube in the maximal tubing corresponding to $T$. The set $B$ corresponding to the subtree rooted at $b$ is a strict subset of $A$, and $(A,B)$ is a nested pair of tubes. Note that $B$ is one of the connected components of $G-a$. The rotation consists of picking $b$ instead of $a$ as the root of the subtree for the graph $G[A]$. The vertex $a$ then becomes the root of the subtree on the connected component $B'$ of $G[A]-b$ that contains $a$. After the rotation, each subtree rooted at a child of $b$ is reattached to either $a$ or $b$, depending on whether the child belongs to the same connected component of $G[A]-b$ as $a$ or not. In terms of maximal tubings, it simply amounts to flipping the tubes $B$ and $B'$.  

\begin{figure}
\begin{center}
\includegraphics[page=4,scale=.6]{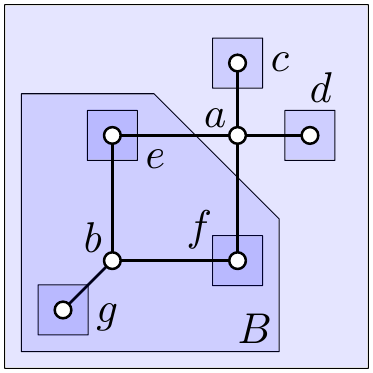}
\includegraphics[page=5,scale=.6]{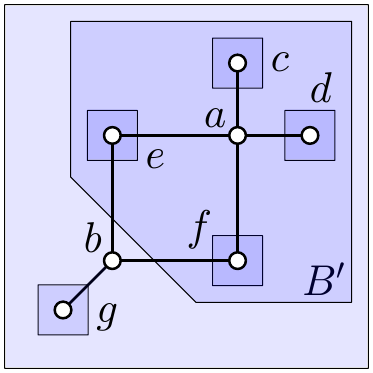}
\quad
\includegraphics[page=6,scale=.6]{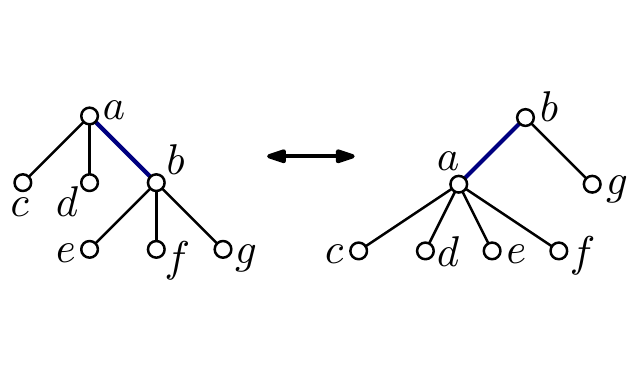}
\end{center}
\caption{\label{fig:rotation}A flip in a tubing, and the corresponding rotation in the associated search tree.}
\end{figure}

The correspondence between flips in tubings and rotations in search trees is illustrated in Figure~\ref{fig:rotation}.

\subsection{Related works}

In a recent paper, Bose, Cardinal, Iacono, Koumoutsos, and Langerman~\cite{BCIKL19} consider the design of competitive algorithms for the problem of searching in trees.
They define a computation model involving search trees on trees, in which pointer moves and rotations all have unit cost.
This can be seen as a generalization of the standard online binary search tree problem~\cite{ST85,W89,DHIP07,DHIKP09}, which has fostered developments in combinatorics, including the exact asymptotic estimate on the diameter of associahedra~\cite{SleatorTarjanThurston1988}.
In the context of search trees on trees, the lower bound of Cardinal, Langerman, and P\'erez-Lantero~\cite{CLP18} on the diameter of tree associahedra ruled out some of the techniques that were known for binary search trees, and motivated the definition of {\em Steiner-closed} search trees. 
The importance of this notion is emphasized in a recent article by Berendsohn and Kozma~\cite{BK20}. 
Our results could give insights on potential generalizations to online search trees on graphs.

Other questions of interest include the Hamiltonicity of graph associahedra, proved by Manneville and Pilaud~\cite{MP15}, and the
computationally efficient generation of Hamilton paths and cycles in their skeleton, recently studied by Cardinal, Merino, and M\"utze~\cite{CMM22}.

\subsection{The diameter of graph associahedra}

We will denote the diameter of the graph associahedron of $G$ by $\delta (\mathcal{A} (G))$. 
Manneville and Pilaud proved the following tight bounds on that quantity as a function of the number of vertices and edges of $G$.
\begin{theorem}[Manneville-Pilaud~\cite{MP15}]
  \label{thm:mlb}
For any connected graph $G$ on $n$ vertices and $m$ edges, the diameter of the graph associahedron of $G$ satisfies
$$
\max\{m, 2n - 20\} \leq \delta (\mathcal A(G)) \leq {n\choose 2}.
$$
\end{theorem}
In order to distinguish between these two extreme cases (linear versus quadratic diameter), we aim at bounds expressed as a function of some graph invariants, or bounds that hold for other families of graphs.

The following result, also from Manneville and Pilaud, will be useful as well .

\begin{theorem}[Manneville-Pilaud~\cite{MP15}]
  \label{thm:monot}
  The diameter $\delta (\mathcal{A}(G))$ is non-decreasing: $\delta (\mathcal{A}(G))\leq \delta(\mathcal{A}(G'))$ for any two graphs $G,G'$ such that $G$ is a subgraph of $G'$. 
\end{theorem}

\subsection{Pathwidth, treewidth, and tree-depth}

We consider three classical numerical invariants of a graph $G$.
We refer the reader to the texts of Diestel~\cite{D05} and Ne{\v s}et{\v r}il and Ossona de Mendez~\cite{NO12} for details and alternative definitions of those parameters.

The {\em pathwidth} of a graph $G$ is $\omega - 1$, where $\omega$ is the smallest clique number of an interval supergraph of $G$, that is, an interval graph that can be obtained from $G$ by adding edges. Similarly, the {\em treewidth} of a graph $G$ is exactly one less than the smallest clique number of a chordal supergraph of $G$.
Pathwidth and treewidth can also be defined in terms of path and tree decompositions, respectively. Paths have pathwidth one, trees have treewidth one, and on an intuitive level, those two parameters quantify how close the graph is to a path or a tree. They also play a key role in the theory of graph minors and in graph algorithms.

The {\em tree-depth} of a graph $G$ is the smallest height of a search tree on $G$, where a tree composed of a single vertex has height one.
The tree-depth is definitely a natural invariant to consider, as it is a function of the exact same objects that form the vertices
of the graph associahedron. Surprisingly, this connection does not seem to have been exploited in previous works.

\subsection{Our Results}

We first prove that the lower bound of $m$ from Manneville and Pilaud on the diameter of the associahedra of a graph on $m$ edges is essentially tight for all
{\em trivially perfect} graphs. Those graphs appear naturally here, as they are maximal for a fixed tree-depth. In Section~\ref{sec:tp}, we properly define trivially perfect graphs, and prove the following result.

\begin{theorem}
\label{teo1tp}
Let $G$ be a connected trivially perfect graph with $m$ edges.
Then $\delta(\mathcal{A}(G)) = \Theta(m)$. 
\end{theorem}

In Section~\ref{sec:td}, we refine the bounds on the diameter of graph associahedra whose underlying graphs have bounded tree-depth or treewidth. Given a family $\calG$ of graphs, we consider the worst-case diameter
$$
\delta_{\calG} (n) = \max_{G\in \calG : |V(G)|=n} \delta (\mathcal {A} (G))  
$$
of their graph associahedra. The tree-depth is an example of parameter that precisely controls the behavior of the diameter of the associahedra, in the following worst-case sense.
\begin{theorem}
\label{thm:tdtight}
Let $\calG$ be the family of graphs on $n$ vertices and of tree-depth at most $\td (n)$. Then 
$$
\delta_{\calG} (n) = \Theta (\td (n)\cdot n).
$$
\end{theorem}

We obtain the following lower and upper bounds as a function of the treewidth of the graph.
\begin{theorem}
  \label{thm:twb}
Let $\calG$ be the family of graphs on $n$ vertices and of treewidth at most $\tw (n)$. Then
$$
\Omega (\tw (n)\cdot n) \leq \delta_{\calG} (n) \leq O (\tw (n)\cdot n\log n).
$$
\end{theorem}

The same bounds hold as a function of the pathwidth of the graph.
\begin{theorem}
  \label{thm:pwb}
Let $\calG$ be the family of graphs on $n$ vertices and of pathwidth at most $\pw (n)$. Then
$$
\Omega (\pw (n)\cdot n) \leq \delta_{\calG} (n) \leq O(\pw(n)\cdot n\log n).
$$  
\end{theorem}

The case of graphs of pathwidth at most {\em two} is intriguing, as one could have suspected that the diameter is close to that of the classical associahedra (the {\em path} associahedra, whose diameter is linear).
In fact, the diameter jumps from linear to linearithmic, as we will show in Section~\ref{sec:pw}.
\begin{theorem}
  \label{thm:pw2tight}
Let $\calG$ be the family of graphs of pathwidth two. Then
$$
\delta_{\calG} (n) = \Theta(n\log n).
$$
\end{theorem}

In Section~\ref{sec:split}, we consider the diameter of complete split graph associahedra. The complete split graph $\mathrm{SPK}_{p,q}$ is a graph whose vertex set can be partitioned into a subset $P$ of $p$ vertices and a subset $Q$ of $q$ vertices inducing a clique and an independent set, respectively, in such a way that every vertex of $P$ is connected by an edge to every vertex of $Q$. In particular, the number of edges of $\mathrm{SPK}_{p,q}$ is precisely
$$
m=pq+{p\choose2}.
$$

We give the exact diameter of the associahedra of these graphs.
\begin{theorem}\label{thm:csga}
If $q \geq 4p+1$ then,
$$
\delta(\mathcal{A}(\mathrm{SPK}_{p,q})) = 2pq + {p\choose 2} = 2m - {p\choose 2}
$$
and otherwise,
$$
\delta(\mathcal{A}(\mathrm{SPK}_{p,q})) = pq+\left\lfloor\frac{1}{2}{q\choose 2}\right\rfloor+{p\choose 2}\leq 2m.
$$
\end{theorem}

We recover the diameter of stellohedra~\cite{MP15} as a special case of Theorem \ref{thm:csga}.

Finally, in Section~\ref{sec:cb}, we provide a general upper bound on the diameter of the associahedron of the complete bipartite graph $\mathrm{K}_{p,q}$ and show that this bound is tight when the graph is sufficiently unbalanced.

\begin{theorem}
\label{teo1-sp}
If $q \geq 4p+1$, then $\delta(\mathcal{A}(\mathrm{K}_{p,q})) = 2pq$.
\end{theorem}

\section{Associahedra of trivially perfect graphs}
\label{sec:tp}

A graph is {\em trivially perfect} if it is both a {\it cograph} and an {\it interval graph}. Wolk \cite{Wol62} called these graphs {\it comparability graphs of trees} and gave characterizations for them. Golumbic \cite{Gol78} called them {\it trivially perfect graphs} because it is trivial to show that such a graph is a {\it perfect} graph. These graphs have the property that in each of their induced subgraphs, the size of the maximum independent set is also the number of maximal cliques. 

\subsection{Trivially perfect graphs}

A {\it universal vertex} in a graph $G$ is a vertex which is adjacent to all the other vertices in $G$. A {\it maximal universal clique} in a graph $G$ is a maximal clique $C$ in $G$ such that each vertex in $C$ is a universal vertex in $G$. Yan et al. \cite{Yan96} give some equivalent definitions of trivially perfect graphs. In particular, they define trivially perfect graph recursively as follows: $(1)$ an isolated vertex (i.e. $\mathrm{K}_1$) is a trivially perfect graph, $(2)$ adding a new universal vertex to a trivially perfect graph results in a trivially perfect graph, and $(3)$ the disjoint union of two trivially perfect graphs is a trivially perfect graph. Drange et al. \cite{Dra15} give the following decomposition of a trivially perfect graph. Let $T$ be a rooted tree and $t$ a vertex of $T$. We denote by $T_t$ the maximal subtree of $T$ rooted at $t$.

\begin{definition}{\bf(Definition 2.3 in \cite{Dra15})}
Consider a trivially perfect graph $G= (V,E)$. A {\it universal clique decomposition} of $G$ is a pair $(T = (V_T,E_T), \mathcal{B} = \{B_t\}_{t \in V_T})$, where $T$ is a rooted tree and $\mathcal{B}$ is a partition of the vertex set $V$ into pairwise disjoint, nonempty subsets, such that
\begin{itemize}
\item if $vw \in E$, $v \in B_t$, and $w \in B_s$, then $s$ and $t$ are on a path from a leaf to the root (and, possibly $s=t$), and
\item if $t \in V_T$, then $B_t$ is the maximal universal clique in the subgraph of $G$ induced by
$$
\bigcup_{s\in V(T_t)}\!\!\!\!\!B_s.
$$
\end{itemize}
\end{definition}

The vertices of $T$ are called {\it nodes} and the sets in $\mathcal{B}$ {\it bags} of the universal clique decomposition $(T,\mathcal{B})$. Note that in a universal clique decomposition, every nonleaf node $t$ has at least two children, since otherwise the universal clique contained in the bag corresponding to $t$ would not be maximal.

Drange et al. \cite{Dra15} have shown that a connected graph $G$ admits a universal clique decomposition if and only if it is trivially perfect. Moreover, such a decomposition is unique up to isomorphism. Figure \ref{fig1tp} shows a trivially perfect graph (left) and its universal clique decomposition (center). It is well known that the tree-depth of a graph $G$ is the minimum size of the largest clique in a trivially perfect supergraph of $G$ (see \cite{Neo15}).

\begin{figure}[htp]
\centering{
\includegraphics[scale=.6]{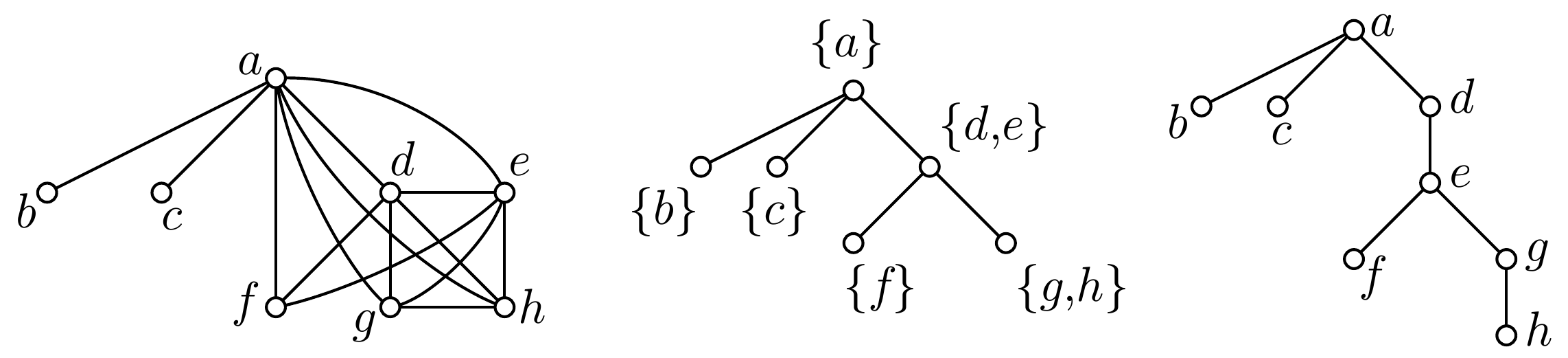}
}
\caption{A trivially perfect graph $G$ (left), the universal clique decomposition of $G$ (center), and a minimum height search tree on $G$ constructed from its universal clique decomposition (right).}
\label{fig1tp}
\end{figure}

\subsection{Upper bound}

Consider a connected trivially perfect graph $G = (V,E)$ with clique number $\omega$. One can use the universal clique decomposition $(T,\mathcal{B})$ of $G$ in order to construct a search tree $T'$ of $G$ whose height is equal to the tree-depth of $G$ and, therefore to $\omega$: let $r$ be the root of $T$ and denote by $r_1$ to $r_p$ its child nodes. Form a path $P$ of length $|B_r|$ whose vertices are the elements of $B_r$ (arranged in any order). One of the ends of this path will be the root of $T'$. Now, denote by $C_i$ the connected component of $G[V \mathord{\setminus} B_r]$ that admits $r_i$ as a subset of its vertices. The pair $(T_{r_i}, \{B_j : j \in V(T_{r_i})\})$ turns out to be the universal clique decomposition of $C_i$. Therefore, one can use the procedure recursively in order to build a search tree for each $C_i$. Connecting the root of these search trees by an edge to one of the ends of $P$ results in the announced search tree $T'$ of $G$ with height $\omega$. This procedure is illustrated in Figure \ref{fig1tp}, where the minimum height search tree is shown on the right.

\begin{theorem}
\label{teo2tp}
Let $G$ be a connected trivially perfect graph on $m$ edges. Then 
$\delta(\mathcal{A}(G)) \leq 2m$.
\end{theorem}
\begin{proof}
Let $td(G) = \omega(G) = k$. We prove by induction on $k$ that any search tree on $G$ can be transformed into a search tree of height $k$ in at most $m$ rotations. When $k = 2$, the graph $G$ has a single edge and the result is immediate. Let $S$ be a search tree on $G$ of height $k$ with root $r$ and denote by $r_1$ to $r_p$ the child vertices of $r$ in $S$. Such a search tree can be obtained from the universal clique decomposition tree of $G$ as described above. Let $T$ be another search tree on $G$. Then, there exists a sequence of at most $n-1$ rotations that transform $T$ into a tree $P$ where $r$ has been lifted at the root. Clearly, $r$ belongs to a maximal universal clique in $G$ and thus, $G \mathord{\setminus} r$ has $p$ connected components, say $C_1$ to $C_p$, each inducing a trivially perfect subgraph of $G$ with clique number at most $k-1$. By the definition of the universal clique decomposition, each subtree $S_{r_i}$ is a search tree of height at most $k-1$ on $G[C_i]$. Denote by $s_1$ to $s_p$ the child vertices of $r$ in $P$. By induction, there exists a sequence of at most $|E(G[C_i])|$ rotations that transform the tree $P_{s_i}$, rooted at $s_i$, into the subtree $T_{r_i}$.

Therefore, $T$ can be transformed into $S$ in at most
$$
n-1+\sum_{1 \leq i \leq p}|E(G[C_i])| = m
$$
rotations. Since any search tree on $G$ can be transformed into $S$ in at most $m$ rotations, $\delta(\mathcal{A}(G)) \leq 2m$.
\end{proof}

Theorem~\ref{teo1tp} is a consequence of Theorems~\ref{thm:mlb} and~\ref{teo2tp}. 

\section{Diameter, tree-depth, and treewidth}
\label{sec:td}
In this section, we establish tight bounds on the diameter of graph associahedra in terms of the tree-depth of the underlying graph. We will make use of the following.

\begin{lemma}
  \label{lem:mtp}
  Let $H$ be a trivially perfect graph on $n$ vertices and $m$ edges, with tree-depth $\td$. Then $m < \td\cdot n$.
\end{lemma}
\begin{proof}
This is easily proved by induction. First observe that, if $\td=1$, then $n=1$ and $m=0$. Now assume that the statement holds for all graphs $H$ with less than $n$ vertices and tree-depth less than $\td$. If $r$ is the root of a search tree $T$ on $H$ of height $\td$, then $m=n-1+|E(H-r)|$. By induction,
$$
|E(H-r)|< (\td -1)(n-1).
$$

As a consequence, $m\leq \td\cdot (n-1) < \td\cdot n$, as desired.
  \end{proof}

\begin{theorem}
  \label{thm:tdub}
Let $G$ be a graph on $n$ vertices, of tree-depth at most $\td$. Then
$\delta (\mathcal A (G)) \leq 2\cdot\td\cdot n$.
\end{theorem}
\begin{proof}
By definition, there exists a trivially perfect supergraph $H$ of $G$ with clique number $\td$. By Lemma~\ref{lem:mtp}, $H$ has at most $\td\cdot n$ edges. Hence, Theorems~\ref{thm:monot} and \ref{teo2tp} yield $\delta(\mathcal{A}(G))\leq \delta(\mathcal{A}(H)) \leq 2\cdot\td\cdot n$.
\end{proof}

Let us prove that this bound is tight up to a constant factor for a wide range of values of $\td(G)$.
\begin{theorem}
  \label{thm:tdlb}
For any two positive integers $k$ and $n$ such that $k$ divides $n$, there exists a trivially perfect graph $G$ on $n+1$ vertices such that $\td (G) = n/k + 1$ and $\delta (\mathcal{A} (G))\geq \td (G)\cdot n/2$. 
\end{theorem}
\begin{proof}
Consider the graph $G$ composed of $k$ cliques $C_1$, $C_2$, \ldots, $C_k$, each of size $n/k+1$, such that a designated vertex $v$ is the unique vertex common to $C_i$ and $C_j$ when $i$ and $j$ are distinct.

Clearly, $G$ is trivially perfect, $\td (G) = n/k + 1$, and the number of edges of $G$ is
$$
k\cdot {n/k+1\choose 2}=\frac{n^2}{2k}+\frac{n}{2}=td(G)\cdot\frac{n}{2}.
$$

The desired bound on the diameter of $\mathcal{A}(G)$ therefore follows from the lower bound stated by Theorem~\ref{thm:mlb}.
\end{proof}

From this construction, we obtain families of polyhedra parameterized by a function $\td (n)$ that interpolate between the stellohedron (a star has tree-depth two) and the permutohedron (a complete graph has tree-depth $n$).

One obtains Theorem~\ref{thm:tdtight} by combining Theorems~\ref{thm:tdub} and \ref{thm:tdlb}.

\begin{corollary}
  \label{cor:twub}
  If $G$ has treewidth at most $\tw$, then
  $\delta (\mathcal A(G)) \leq c\cdot\tw \cdot n\log n$
  for some constant $c$.
\end{corollary}
\begin{proof}
This follows from Theorem~\ref{thm:tdub} and the known fact that $\td (G)=O(\tw(G) \cdot\log n)$~\cite{NO12}.
\end{proof}

\section{Diameter and pathwidth}
\label{sec:pw}

We first prove a lower bound on the diameter of graph associahedra in terms of pathwidth.
\begin{theorem}
  \label{thm:pwcliques}
  For any $k\geq 2$ and any $n$ multiple of $k-1$, there exists an interval graph on $n+1$ vertices and of pathwidth $k-1$ such that $\delta (\mathcal{A} (G))$ is at least $nk/2$.
\end{theorem}
\begin{proof}
  We consider a graph $G$ induced by a sequence of cliques $C_1$, $C_2$, \ldots, $C_{n/(k-1)}$, each of size $k$, such that any two consecutive cliques $C_i$ and $C_{i+1}$ have a single vertex in common, and no other pairs of cliques have a vertex in common. This graph is clearly an interval graph of pathwidth $k-1$, and its number of edges is
  $$
  {k\choose 2}\frac{n}{k-1}\geq \frac{nk}{2}.
  $$
  The conclusion follows from the lower bound in Theorem~\ref{thm:mlb}.
\end{proof}

Theorem~\ref{thm:pwcliques} and Corollary~\ref{cor:twub} together prove Theorems~\ref{thm:twb} and~\ref{thm:pwb}.

Paths have pathwidth one, and their graph associahedra have linear diameter. 
Interestingly, as we shall see, the diameter jumps to $\Omega (n\log n)$ for graphs of pathwidth two.
Our proof uses a construction similar to the one from Cardinal, Langerman, and P{\'{e}}rez{-}Lantero~\cite{CLP18} for tree associahedra.
We need some preliminaries involving chordal graphs and {\em projections} of rotation sequences.

\subsection{Chordal graphs}

A graph $G$ is {\em chordal} if it does not contain induced cycles of length $4$ or more. In other words, every cycle in $G$ of length $4$ or more has a {\em chord}.
We denote by $N(v)$ the set of the neighbors of a vertex $v$ in $G$.
A vertex $v$ is said to be {\em simplicial} if $G[N(v)]$ is a clique. 
It is known that a graph is chordal if and only if it has a {\em perfect elimination ordering}: an ordering of its vertices such that the set of neighbors of a vertex
$v$ that come after $v$ in the ordering induce a clique. Hence, a perfect elimination ordering is obtained by iteratively removing a simplicial vertex in 
the remaining subgraph. The set of the vertices remaining after removing the vertices in a prefix of a perfect elimination ordering is called {\em monophonically convex}, or {\em m-convex};
see Farber and Jamison~\cite{FJ86} and references therein for details. In what follows, we will simply call such sets of vertices {\em convex}.

\subsection{Projections}

We now introduce a tool that will turn out useful for performing inductions on chordal graphs.

\begin{observation}
Consider a chordal graph $G$ on at least two vertices, a simplicial vertex $v$ of $G$, and a search tree $T$ on $G$. Then $v$ has at most one child in $T$. Further consider the tree $T'$ obtained as follows.
  \begin{enumerate}
  \item If $v$ is a leaf of $T$, just remove $v$ from the tree.
  \item If $v$ is the root of $T$, then remove $v$ from $T$ and designate its child as the new root.
    \item If $v$ has both a parent and a child, then remove $v$ from $T$ and replace the two edges between $v$ and
  its parent and child by a single edge between its parent and its child.
    \end{enumerate}
Then $T'$ is a search tree on $G-v$.
\end{observation}

Given an initial search tree on $G$, we can construct a search tree on $G[S]$ by induction, for any convex subset $S\subseteq V$. Note that the obtained search tree does not depend on the order in which the simplicial vertices have been removed. Letting $T_{|S}$ be the tree thus obtained, we call it the {\em projection} of $T$ on $S$.

\begin{observation}
Let $T$ be a search tree on $G=(V,E)$, and let $T'$ be obtained from $T$ by performing a single rotation.
Then the projections of $T$ and $T'$ on $V\mathord{\setminus} \{v\}$ are either identical or related by a single rotation.
They are identical if and only if the rotation between $T$ and $T'$ involves $v$.
\end{observation}

The following lemma is a consequence of the above observations and generalizes Lemma~3 in~\cite{CLP18}.
\begin{lemma}[Projection Lemma]
  \label{lem:projection}
Let $G=(V,E)$ be a chordal graph, and $\pi$ a sequence of rotations transforming a search tree $T$ on $G$ into another one, say $T'$.
Let $S$ be a convex subset of $V$.
The {\em projection} $\pi_{|S}$ of $\pi$ on $S$ is the sequence of rotations obtained by removing from $\pi$ all rotations involving two
vertices at least one of whose does not belong to $S$.
Then $\pi_{|S}$ is a rotation sequence that transforms $T_{|S}$ into $T'_{|S}$.
\end{lemma}

\subsection{An $\Omega (n\log n)$ lower bound for graphs of pathwidth two}

Let us first recall how {\em bit-reversal} permutations work~\cite{K96,W89,DHIKP09}.
We denote a permutation on $n$ elements by a sequence composed of one occurrence of each of the first $n$ positive integers.
The bit-reversal permutation of length one is $\sigma_1 = 1$.
The $k$th bit-reversal permutation $\sigma_k$ has length $n=2^{k-1}$ and is obtained by concatenating $2\sigma_{k-1}-1$ with $2\sigma_{k-1}$.
Note that the permutation $\sigma_k$ of length $n$ alternates between entries at most $n/2$ and entries greater than $n/2$.
Here are the first five bit-reversal permutations:

\begin{eqnarray*}
\sigma_1 & = & 1 \\
\sigma_2 & = & 1, 2\\ 
\sigma_3 & = & 1, 3, 2, 4\\
\sigma_4 & = & 1, 5, 3, 7, 2, 6, 4, 8\\ 
\sigma_5 & = & 1, 9, 5, 13, 3, 11, 7, 15, 2, 10, 6, 14, 4, 12, 8, 16.
\end{eqnarray*}

\begin{theorem}
\label{thm:pw2lb}
Let $n$ be a power of two. If $n$ is at least $4$, then there exists a graph $G$ on $n$ vertices, with pathwidth two such that $\delta (\mathcal{A} (G))$ is at least $\Omega (n\log n)$.
\end{theorem}
\begin{proof}
 Let $n=2^{k-1}$ for some $k$.
 We construct a graph $G_n$ of pathwidth two composed of $2n$ vertices denoted by $a_1$, $a_2$, \ldots, $a_n$ and $b_1$, $b_2$, \ldots, $b_n$.
 We let $a_i$ be adjacent to $a_{i+1}$ and $b_{i+1}$ when $i<n$. Similarly, $b_i$ is adjacent to $b_{i+1}$ when $i<n$ and $a_i$ to $b_i$ for all $i$. It can be checked that $G_n$ is an interval graph with clique number three, hence of pathwidth two. The graph is shown on Figure~\ref{fig:pw2}. We define the two subsets of vertices 
$$
L=\bigcup_{i\leq n/2}\{a_i,b_i\}
$$
and
$$
R=\bigcup_{i>n/2}\{a_i,b_i\}.
$$

Both of these subsets are convex. Further note that $G[L]$ and $G[R]$ are each isomorphic to $G_{n/2}$.

We now consider the rotation distance between two search trees $T$ and $T'$ on $G_n$. These trees are paths of the following form, rooted at the first vertex:
\begin{center}
\begin{tabular}{crl}
$T$ & {\color{darkblue}$a_1$}, {\color{darkblue}$b_1$}, {\color{darkblue}$a_2$}, {\color{darkblue}$b_2$}, {\color{darkblue}\ldots}, {\color{darkblue}$a_{n/2}$}, {\color{darkblue}$b_{n/2}$}, & {\color{darkred} $a_{n/2+1}$}, {\color{darkred}$b_{n/2+1}$}, {\color{darkred}\ldots}, {\color{darkred}$a_n$}, {\color{darkred}$b_n$} \\
$T'$ & {\color{darkblue}$a_{\sigma_k (1)}$}, {\color{darkred}$a_{\sigma_k (2)}$}, {\color{darkblue}$a_{\sigma_k (3)}$}, {\color{darkred}$a_{\sigma_k (4)}$}, \ldots, {\color{darkblue}$a_{\sigma_k (n-1)}$},  {\color{darkred} $a_{\sigma_k (n)}$}, & {\color{darkblue} $b_1$}, {\color{darkblue}$b_2$}, {\color{darkblue}\ldots}, {\color{darkblue}$b_{n/2}$}, {\color{darkred} $b_{n/2+1}$}, {\color{darkred}\ldots} ,{\color{darkred}$b_n$}
\end{tabular}
\end{center}

The first path, $T$, is indeed a search tree because it corresponds to a perfect elimination ordering. The second path, $T'$, is also a search tree, because the subgraphs corresponding to subtrees rooted at each of the $a_i$ are connected, thanks to the presence of vertices $b_i$. By the recursive definition of bit-reversal permutations, the search trees $T_{|L}$ and $T'_{|L}$ on $G[L]$ are obtained in the exact same way, by permuting the elements of $L$ according to $\sigma_{k-1}$. The same holds for $T_{|R}$ and $T'_{|R}$ on $G[R]$, up to a shift of the indices. 
\begin{figure}[t]
\begin{center}
\includegraphics[page=3,scale=.6]{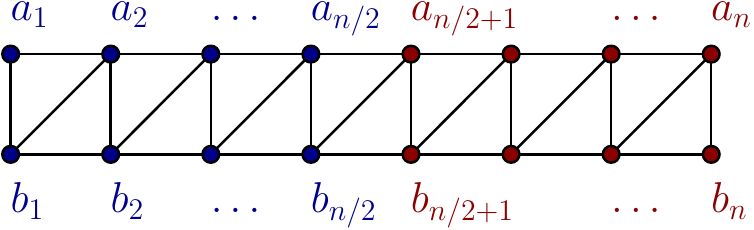}
\end{center}
\caption{\label{fig:pw2}The graph $G_n$ used in the proof of Theorem~\ref{thm:pw2lb}.}
\end{figure}
The sets $L$ and $R$ induce a two-coloring of any search tree on $G_n$. 
An edge of a search tree will be called {\em monochromatic} if both of its endpoints belong to the same set $L$ or $R$, and {\em bichromatic} otherwise. Similarly, we distinguish monochromatic rotations, involving pairs of vertices from the same set $L$ or $R$, from bichromatic rotations involving one vertex of each set.

Let $\pi$ be a sequence of rotations transforming $T$ into $T'$, of minimum length $\ell (n)$.
From our previous observations and Lemma~\ref{lem:projection}, the number of monochromatic rotations in $\pi$ is
$$
|\pi_{|L}| + |\pi_{|R}|\geq 2\ell\left(\frac{n}{2}\right)\!.
$$

We now give a lower bound on the number of bichromatic rotations in $\pi$.
Given a search tree on $G_n$, we define its {\em alternation number} as the maximum, over all paths from the root to a leaf, of the number of bichromatic edges on the path. 
Note that, from the property of bit-reversal permutations, the alternation number of $T'$ is $n+1$. On the other hand, the alternation number of $T$ is 1. We make two observations.

First, a monochromatic rotation cannot increase the alternation number of a tree. 
Indeed, consider a rotation involving vertices $a,b$, with $b$ the child of $a$ (refer to Figure~\ref{fig:rotation}). 
The only edges of the tree whose endpoints are changed by the rotation are the edge from the parent of $a$, 
and edges connecting $b$ to the root of a subtree in the initial tree. 
If $a$ and $b$ have the same color, none of these edges can become bichromatic.

Second, a bichromatic rotation can only increase the alternation number of a search tree by two.
Indeed again, on a path from the root to a leaf, only two edges can become bichromatic by a rotation involving $a$ and $b$: the one from the parent of $a$ and one from $b$ to the root of a subtree.

We conclude that there must be at least $n/2$ bichromatic rotations in $\pi$.
Summing the number of monochromatic and bichromatic rotations, we obtain
$\ell (n) \geq 2\ell (n/2) + n/2  = \Omega (n\log n).$
\end{proof}

\section{Associahedra of complete split graphs}
\label{sec:split}

Let $p$ and $q$ be two positive integers. In this section, we provide the exact diameter of $\mathcal{A}(\mathrm{SPK}_{p,q})$, where $\mathrm{SPK}_{p,q}$ is the complete split graph composed of two disjoint subsets of vertices, a subset $P$ of size $p$ inducing a clique and a subset $Q$ inducing an independent set of size $q$, and having all the edges between these two sets. We let
$$
n=p+q
$$
be the number of vertices of $\mathrm{SPK}_{p,q}$ and
$$
m=pq+{p\choose 2}
$$
be its number of edges. Scheffler~\cite{S93} showed that the tree-depth of $\mathrm{SPK}_{p,q}$ is equal to $p+1$. It is particularly noteworthy that associahedra of complete split graphs interpolate between the stellohedron (when $p=1$) whose diameter is linear and the permutohedron (when $q=1$) whose diameter is quadratic. Hence, by Theorem \ref{thm:mlb}, this is a consistent family of graph associahedra whose diameters range from one extreme behavior to the other.

Let us first observe that search trees on the graph $\mathrm{SPK}_{p,q}$ all look like \textit{brooms}: a chain of vertices, including all vertices
from the clique $P$, attached to a star whose center is one of the vertices of $P$. We refer to the vertices in the initial chain, including the
last vertex forming the center of the star, as the \textit{handle} of the broom. In what follows, we refer to search trees on complete split graphs as brooms.

Figure \ref{fig3star} shows the complete split graph $\mathrm{SPK}_{3,q}$ for some $q \geq 4$, and four brooms on $\mathrm{SPK}_{3,q}$.

\begin{figure}[htp]
\begin{center}
\includegraphics[scale=.6,page=1]{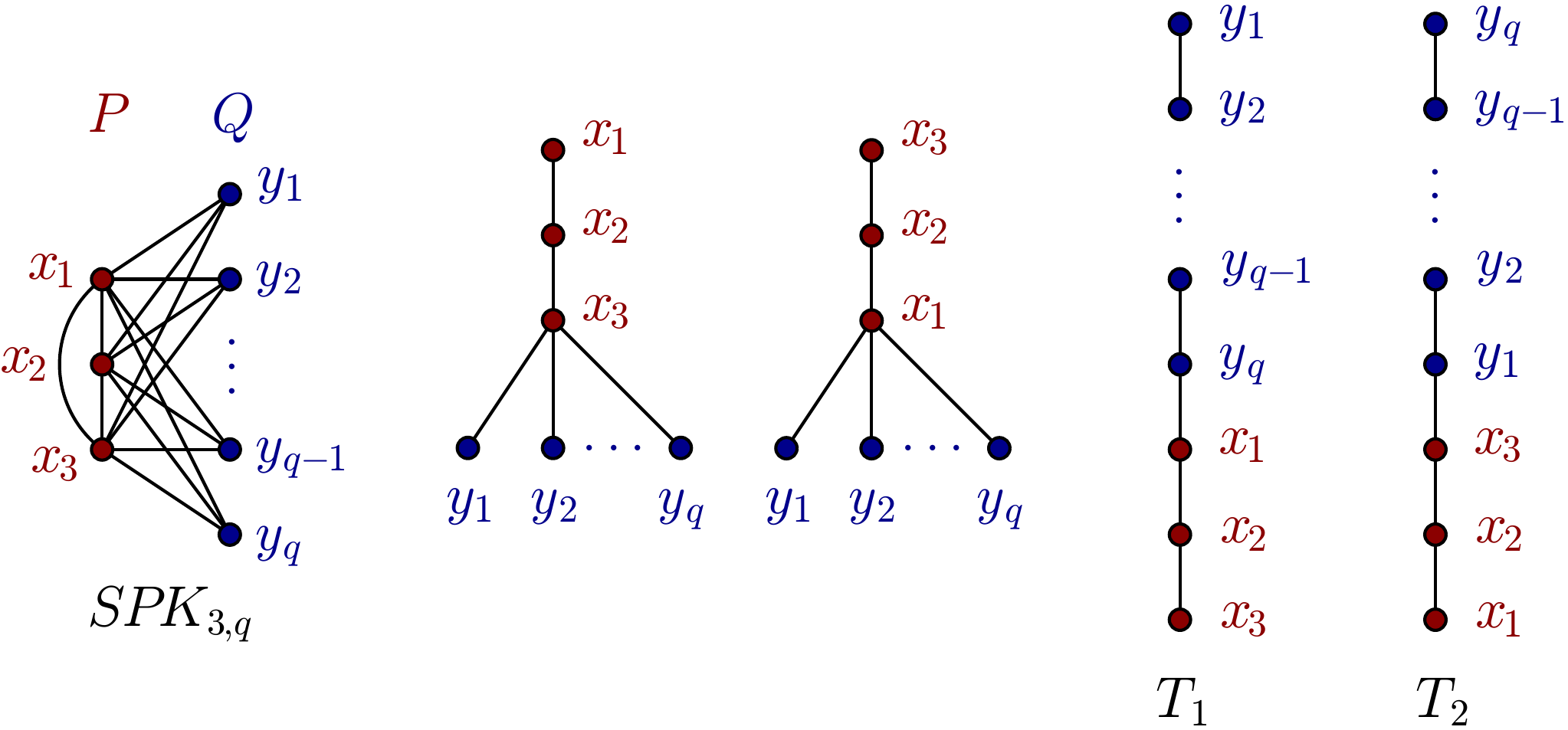}
\end{center}
\caption{The complete split graph $\mathrm{SPK}_{3,q}$ for some $q \geq 4$ and four brooms on $\mathrm{SPK}_{3,q}$.}
\label{fig3star}
\end{figure}

\subsection{Lower bound}

Denote by $x_1$ to $x_p$ the vertices in $P$ and by $y_1$ to $y_q$ the vertices in $Q$. Further consider two non-negative integers $\alpha$ and $\beta$ such that $\alpha<p$ and $\beta<q$.
In order to prove a lower bound for $\delta(\mathcal{A}(\mathrm{SPK}_{p,q}))$, we compute a lower bound for the distance of the following two brooms: $T_1$ having vertices in the order
$$
{\color{darkred} x_1}, {\color{darkred}x_2}, {\color{darkred}\ldots}, {\color{darkred}x_\alpha}, {\color{darkblue}y_1}, {\color{darkblue}y_2}, {\color{darkblue}\ldots}, {\color{darkblue}y_q}, {\color{darkred}x_{\alpha+1}}, {\color{darkred}\ldots}, {\color{darkred}x_p}
$$
and $T_2$ having vertices in the order
$$
{\color{darkblue}y_q},{\color{darkblue}y_{q-1}},{\color{darkblue}\ldots,y_{\beta+1}},{\color{darkred}x_p},{\color{darkblue}y_\beta},{\color{darkblue}\ldots}, {\color{darkblue}y_1}, {\color{darkred}x_{p-1}},{\color{darkred}\ldots},{\color{darkred}x_1}.
$$

Note that, when $\alpha=0$, the root of $T_1$ is $y_1$ and the child of $y_q$ in $T_1$ is $x_1$. Similarly, when $\beta=0$, $x_p$ is the child of $y_1$ in $T_2$ as illustrated in Figure~\ref{fig3star} when $p=3$. When $\beta$ is positive, $x_p$ is the only vertex of $P$ above a vertex of $Q$ in $T_2$. From now on, $\mathrm{dist}(T_1,T_2)$ denotes the rotation distance between $T_1$ and $T_2$. The quantity
$$
f(k)={p \choose 2}+{q\choose 2}-{k\choose 2}+\alpha{q}+k(2p-2\alpha-1)+|\beta-k|
$$
will appear in our estimation of $\mathrm{dist}(T_1,T_2)$. Let us first prove the following technical statement.
\begin{proposition}\label{p1cb}
If $0\leq{k}\leq{q}$, then $f(k)\geq\min\{f(0),f(q)\}$.
\end{proposition}
\begin{proof}
First observe that $f$ is a concave (quadratic) function of $k$ on the interval $[0,\beta]$. It is also a concave function of $k$ on the interval $[\beta,q]$. As a consequence, in order to minimize it, we only need to consider its values when $k$ is equal to $0$, to $\beta$, and to $q$. In other words, it is sufficient to prove that
$$
f(\beta)\geq\min\{f(0),f(q)\}.
$$

Let us reach a contradiction under the assumption that $f(\beta)$ is less than both $f(0)$ and $f(q)$. An immediate consequence of this assumption is that $0<\beta<q$. On the one hand, $f(\beta)<f(0)$ can be rewritten as
$$
-{\beta\choose 2}+\beta(2p-2\alpha-1)<\beta.
$$

Since $\beta$ is positive, dividing this inequality by $\beta$ yields
\begin{equation}\label{p1cbe1}
4(p-\alpha)-2\leq\beta.
\end{equation}

On the other hand, $f(\beta)<f(q)$ can be rewritten as
$$
-{\beta\choose 2}+\beta(2p-2\alpha-1)<-{q\choose 2}+q(2p-2\alpha-1)+q-\beta.
$$

Reorganising the terms of this inequality yields
\begin{equation}\label{p1cbe2}
{q\choose 2}-{\beta\choose 2}-2(q-\beta)(p-\alpha)<0.
\end{equation}

Observe that
$$
{q\choose 2}-{\beta\choose 2}=(q-\beta)\frac{q+\beta-1}{2}.
$$
and recall that $q-\beta$ is positive. Hence, dividing (\ref{p1cbe2}) by $q-\beta$, one obtains
$$
\beta\leq4(p-\alpha)-q.
$$

Combining this inequality with (\ref{p1cbe1}) results in the following estimate of $\beta$.
\begin{equation}\label{p1cbe3}
4(p-\alpha)-2\leq\beta\leq4(p-\alpha)-q.
\end{equation}

As an immediate consequence, $q$ is at most $2$. Since $\beta$ and $q$ are integers satisfying $0<\beta<q$, this proves that $q$ is equal to $2$ and $\beta$ to $1$. In turn, it follows from (\ref{p1cbe3}) that
$$
4(p-\alpha)=3,
$$
which is impossible because p and $\alpha$ are integers.
\end{proof}

Using Proposition \ref{p1cb}, we can prove the following bound on $\mathrm{dist}(T_1,T_2)$.

\begin{lemma}
\label{l2cb}
$\displaystyle\mathrm{dist}(T_1,T_2)\geq {p\choose 2}+\alpha{q}+\min\left\{{q\choose 2}+\beta,2q(p-\alpha)-\beta\right\}$.
\end{lemma}
\begin{proof}
Consider a shortest sequence of rotations transforming $T_1$ into $T_2$, and the corresponding sequence of brooms.
We denote by $L\subseteq Q$ the set of vertices of the independent set that appear at least once as a leaf of a broom in this sequence, and let $k=|L|$. Hence, the handle of every broom in the sequence has at least $p + q - k$ vertices.

We first lower bound the number of rotations involving two vertices of $P$ or two vertices of $Q\mathord{\setminus}L$. Since the relative order of any such pair of vertices is reversed between $T_1$ and $T_2$, there must be at least
$$
{p \choose 2}+{q-k\choose 2}
$$
rotations involving them. Now observe that the order of a vertex in $Q\mathord{\setminus}L$ and any of the vertices $x_1$ to $x_\alpha$ is also reversed between $T_1$ and $T_2$. Hence there must be at least $\alpha(q-k)$ additional rotations involving any such pair of vertices. Moreover, if $k\leq\beta$, then at least $\beta-k$ vertices from $Q\mathord{\setminus}L$ must be exchanged with $x_p$.

As a consequence, there are at least
$$
\alpha(q-k)+\max\{0,\beta-k\}
$$
rotations involving a vertex in $Q\mathord{\setminus}L$ and $x_p$.

We can also give a lower bound on the number of rotations involving a vertex from $L$ and a vertex from $P$. Each vertex of $L$ must go at least once below $x_{\alpha+1}$ to $x_p$ when moving down from their position in $T_1$ in order to become a leaf, which costs at least $k(p-\alpha)$ rotations. Similarly, each vertex of $L$ must go at least once above $x_1$ to $x_{p-1}$ when moving up to their position in $T_2$. Moreover, if $k>\beta$, there must be at least $k-\beta$ vertices from $L$ that each go once above $x_p$. This costs an additional $k(p-1)+\max\{0,k-\beta\}$ rotations, and the number of rotations involving a vertex from $L$ and a vertex from $P$ is then at least
$$
k(2p-\alpha-1)+\max\{0,k-\beta\}
$$

Every vertex of $L$ must also be swapped with every vertex in $Q\mathord{\setminus}L$ in the handle of the broom, 
either when moving down to the set of leaves, or back up in the handle, which costs $k(q-k)$ rotations.

Overall, the number of rotations in the considered sequence is at least
$$
{p \choose 2}+{q-k\choose 2}+\alpha(q-k)+\max\{0,\beta-k\}+k(2p-\alpha-1)+\max\{0,k-\beta\}+k(q-k)
$$
which turns out to be exactly $f(k)$. Therefore, by Proposition \ref{p1cb},
$$
\mathrm{dist}(T_1,T_2)\geq\min\{f(0),f(q)\}.
$$

Observing that
$$
\left\{
\begin{array}{l}
\displaystyle f(0)={p\choose2}+{q\choose2}+\alpha{q}+\beta,\\[\bigskipamount]
\displaystyle f(q)={p\choose2}+\alpha{q}+2q(p-\alpha)-\beta,\\
\end{array}
\right.
$$
completes the proof.
\end{proof}

By choosing appropriate values for $\alpha$ and $\beta$, we get the following.

\begin{lemma}\label{l3cb}
If $q\geq4p+1$ then
$$
\delta(\mathcal{A}(\mathrm{SPK}_{p,q}))\geq 2pq+{p\choose 2}
$$
and otherwise,
$$
\delta(\mathcal{A}(\mathrm{SPK}_{p,q}))\geq pq+\left\lfloor\frac{1}{2}{q\choose 2}\right\rfloor+{p\choose 2}\!.
$$
\end{lemma}
\begin{proof}
First pick $\alpha=\beta=0$. According to Lemma \ref{l2cb},
$$
\mathrm{dist}(T_1,T_2)\geq {p\choose 2}+\min\left\{{q\choose 2},2pq\right\}.
$$

Further assume that $q\geq4p+1$. In this case,
$$
2pq\leq{q\choose2}
$$
and we obtain the desired bound on $\delta(\mathcal{A}(\mathrm{SPK}_{p,q}))$. Now assume that $q\leq4p$. Observe that, if $q$ is equal to $1$, then $\mathrm{SPK}_{p,q}$ is just a complete graph on $p+1$ vertices and its graph associahedron is the permutohedron of diameter
$$
{p+1\choose2}=p+{p\choose2}.
$$
Hence the announced bound holds in this case, and we can assume that $q$ is at least $2$. Denote
$$
\gamma=\frac{4p+1-q}{4}.
$$

Pick $\alpha=\lfloor\gamma\rfloor$ and $\beta=\lfloor{q(\gamma-\lfloor\gamma\rfloor)}\rfloor$. Note that $\alpha$ is non-negative because $q$ is at most $4p$. In addition, $\alpha$ must be less than $p$ because $q$ is at least $2$ and by construction, $\beta$ is then a non-negative integer less than $q$. In other words, $\alpha$ and $\beta$ satisfy the requirements we imposed on them when defining $T_1$ and $T_2$.

Further observe that $\alpha{q}+\beta=\lfloor{q\gamma}\rfloor$ and that
$$
q\gamma=pq-\frac{1}{2}{q\choose2}.
$$

As a consequence,
$$
\alpha{q}+\beta=pq-\left\lceil\frac{1}{2}{q\choose2}\right\rceil
$$
and in turn,
\begin{equation}\label{l3cbe1}
\left\{
\begin{array}{l}
\displaystyle\!{q\choose2}+\alpha{q}+\beta=pq+\left\lfloor\frac{1}{2}{q\choose2}\right\rfloor\!\!,\\[\bigskipamount]
\displaystyle2pq-\alpha{q}-\beta={pq}+\left\lceil\frac{1}{2}{q\choose2}\right\rceil\!\!.\\
\end{array}
\right.
\end{equation}

However, it follows from Lemma \ref{l2cb} that
$$
\displaystyle\mathrm{dist}(T_1,T_2)\geq{p\choose 2}+\min\left\{{q\choose 2}+\alpha{q}+\beta,2pq-\alpha{q}-\beta\right\}\!.
$$

Combining this with (\ref{l3cbe1}) completes the proof.
\end{proof}

\subsection{Upper bound}

\begin{lemma}
\label{l4cb}
Let $p$ and $q$ be two positive integers. 
Then
$$
\delta(\mathcal{A}(\mathrm{SPK}_{p,q})) \leq 2pq + {p\choose 2} = 2m - {p\choose 2}.
$$
\end{lemma}
\begin{proof}
Let $P$ and $Q$ be the bipartition of the vertex set of $\mathrm{SPK}_{p,q}$. Let $T$ and $T'$ be any two brooms on $\mathrm{SPK}_{p,q}$. Let $\pi_T$ (respectively, $\pi_{T'}$) be the permutation of the vertex set $P$ corresponding to the order in which the vertices appear in $T$ (respectively $T'$). Let $S_T$ (respectively $S_{T'}$) be the broom in which the first $p$ vertices are the vertices of $P$ under the permutation $\pi_T$ (respectively $\pi_{T'}$) and having the vertices of $Q$ as leaves. 
Observe that $T$ can be transformed into $S_T$ by at most $pq$ rotations. 
Similarly, $T'$ can be transformed into $S_{T'}$ by at most $pq$ rotations. 
By the observation that $S_T$ can be transformed into $S_{T'}$ with at most ${p\choose 2}$ rotations, the desired bound holds.
\end{proof}

We complement Lemma \ref{l4cb} with a stronger bound in the case when $q\leq4p$.
\begin{lemma}
\label{l5cb}
Assume that $q\leq4p$. Then
$$
\delta(\mathcal{A}(\mathrm{SPK}_{p,q})) \leq pq+\left\lfloor\frac{1}{2}{q\choose 2}\right\rfloor+{p\choose 2}\!.
$$
\end{lemma}
\begin{proof}
Consider two brooms $T$ and $T'$ be any two brooms on $\mathrm{SPK}_{p,q}$. If $1\leq{i}\leq{q}$, denote by $w_i$ the number of vertices of $P$ above $y_i$ in $T$. For instance, if $y_i$ is a leaf of $T$, then $w_i=p$. Similarly, let $w_i'$ be the number of vertices of $P$ above $y_i$ in $T'$. Let us now build two different paths from $T$ to $T'$. In the first path, all the vertices of $Q$ in the handle of $T$ are first moved down to the leaves, which takes exactly
$$
pq-\sum_{i=1}^qw_i
$$
rotations. Doing the same in $T'$ takes 
$$
pq-\sum_{i=1}^qw_i'
$$
rotations. The two resulting brooms can then be transformed into one another by just sorting the $p$ vertices of $P$ that remain in their handle, hence producing a path of length at most
$$
{p\choose2}+2pq-\sum_{i=1}^qw_i+w_i'.
$$

In the second path, all the vertices of $Q$ (including the ones that are leaves) of $Q$ are moved up in such a way that all the vertices of $P$ are below all the vertices of $Q$ within the handle of the broom. We assume here that these moves never exchange two vertices of $Q$. Therefore, this takes exactly
$$
\sum_{i=1}^qw_i
$$
rotations. Doing the same in $T'$ takes another
$$
\sum_{i=1}^qw_i'
$$
rotations. Now, the two resulting brooms can be changed into one another by sorting the vertices of $P$ and the vertices of $Q$ separately within their handles resulting in a path of length at most
$$
{p\choose2}+{q\choose2}+\sum_{i=1}^qw_i+w_i'.
$$

We have therefore proven that
$$
\mathrm{dist}(T,T')\leq\min\left\{{p\choose2}+2pq-W,{p\choose2}+{q\choose2}+W\right\}
$$
where $W$ is the sum of all $w_i$ and all $w_i'$. Now observe that this bound is largest possible when
$$
W=pq-\frac{1}{2}{q\choose2}.
$$

As a consequence,
$$
\mathrm{dist}(T,T')\leq{pq}+\frac{1}{2}{q\choose2}+{p\choose2},
$$
which immediately provides the desired upper bound on $\delta(\mathcal{A}(\mathrm{SPK}_{p,q}))$.
\end{proof}

We obtain Theorem \ref{thm:csga} by combining Lemmas~\ref{l3cb}, \ref{l4cb}, and \ref{l5cb}. Note that $\mathcal{A}(\mathrm{SPK}_{1,n})$ is the stellohedron,
the associahedron of the star $S_{1,n}$. Therefore, we recover the following result as a special case of Theorem~\ref{thm:csga}.

\begin{corollary}[Manneville-Pilaud~\cite{MP15}]
\label{cor:stello}
Let $n\geq 5$ be an integer. Then, $\delta(\mathcal{A}(\mathrm{S}_{1,n})) = 2n$.
\end{corollary}

\section{Associahedra of complete bipartite graphs}
\label{sec:cb}

In this section, we study the diameter of $\mathcal{A}(\mathrm{K}_{p,q})$, where $\mathrm{K}_{p,q}$ is the complete bipartite graph composed by two disjoint independent sets of vertices, $P$ of size $p$ and $Q$ of size $q$, and having all the edges between these two sets.
As in the previous section, we denote by $x_1$ to $x_p$ the vertices of $P$, and by $y_1$ to $y_q$ the vertices of $Q$.
We note again that the search trees on $\mathrm{K}_{p,q}$ are brooms, whose handle fully contains one of the two subsets $P,Q$.
Scheffler~\cite{S93} showed that the tree-depth of $\mathrm{K}_{p,q}$ is equal to $\min\{p,q\}+1$. 
Erokhovets~\cite{E09} also considered these polytopes and proved that they satisfy Gal's conjecture, extending a result from Postnikov et al.~\cite{PW08}.

\subsection{Lower bound}

In order to prove a lower bound for $\delta(\mathcal{A}(\mathrm{K}_{p,q}))$, we compute a lower bound for the distance of two brooms. The handle of the first broom, $T_1$, is made up of the vertices in $Q$ in the order
$$
{\color{darkblue}y_1},{\color{darkblue}y_2},{\color{darkblue}\ldots},{\color{darkblue}y_q}
$$
and its leaves are exactly the vertices in P. The second broom, $T_2$, is defined just as $T_1$ except that the order of the vertices in the handle is
$$
{\color{darkblue}y_q},{\color{darkblue}y_{q-1}},{\color{darkblue}\ldots},{\color{darkblue}y_1}
$$

These brooms are depicted in Figure \ref{fig4star} when $p=3$. 
\begin{figure}[htb]
\begin{center}
\includegraphics[scale=.6,page=2]{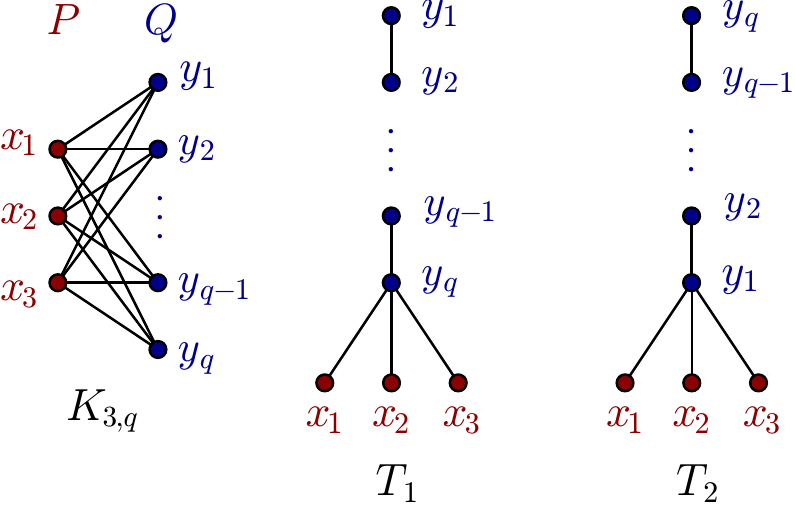}
\end{center}
\caption{The complete bipartite graph $\mathrm{K}_{3,q}$ for some $q \geq 4$ and brooms on $\mathrm{K}_{3,q}$.}
\label{fig4star}
\end{figure}

\begin{lemma}
\label{l2sp}
Let $p,q$ be positive integers with $q \geq 4p+1$. 
Then $\mathrm{dist}(T_1,T_2) \geq 2pq$.
\end{lemma}
\begin{proof}
The proof proceeds as that of Lemma \ref{l2cb}. We consider a shortest sequence of rotations that change $T_1$ into $T_2$, and the corresponding sequence of brooms.
We denote by $L$ the subset of the vertices in $Q$ that appear at least once as a leaf of a broom in this sequence and by $k$ the number of these vertices.

Since the $q-k$ vertices in $Q\mathord{\setminus}L$ remain in the handle of all the brooms in the considered sequence, and the relative order of two such vertices is inverted in $T_1$ and $T_2$, there must be at least
$$
q-k\choose2
$$
rotations involving two of them along that sequence. Similarly, there must be at least $k(q-k)$ rotations exchanging a vertex from $L$ with a vertex from $Q\mathord{\setminus}L$ along the handle of a broom in the considered sequence because the relative orders of these vertices are inverted between $T_1$ and $T_2$.

Now recall that all the brooms in the considered sequence must contain the whole of $P$ or the whole of $Q$ in their handle. Hence, if $k$ is positive, then all the vertices from $P$ must be lifted into the handle before a single vertex from $Q$ becomes a leaf of the broom. As the vertices of $P$ are always inserted at the bottom of the handle by rotations, at least two rotations exchange each vertex in $P$ and each vertex in $L$ within the handle of the broom: one when the latter vertex moves down in order to become a leaf and another when it moves back up. Hence there must be at least $2pk$ rotations involving a vertex from $P$ and a vertex from $L$.

The total number of rotations along the considered sequence is then at least
$$
2pk+k(q-k)+{q-k\choose2}.
$$

This quantity is a concave function of $k$ equal to $q \choose 2$
when $k=0$ and to $2pq$ when $k=q$. As
$$
2pq\leq{q \choose 2}
$$
when $q\geq4p+1$, the rotation distance between $T_1$ and $T_2$ is at least $2pq$ in this case, as desired.
\end{proof}

\subsection{Upper bound}

Before we prove an upper bound for $\delta(\mathcal{A}(\mathrm{K}_{p,q}))$, we introduce some definitions and technical results. 

\begin{definition}
Let $p$ and $q$ be two positive integers and $T$ a search tree on $\mathrm{K}_{p,q}$. We denote by $N_T$ (respectively $X_T$) the search tree on $\mathrm{K}_{p,q}$ having as leaves all the vertices in $P$ (respectively, all the vertices in $Q$) and where its $q$ (respectively~$p$) first vertices are in the same order that they appear in $T$ from the root to the leaves, breaking ties arbitrary when more than one vertex has the same height in $T$.
\end{definition}

\begin{lemma}
\label{lemme-dist-sp}
Let $p,q$ be positive integers and let $T$ be a broom on $\mathrm{K}_{p,q}$. Then 
$$
\mathrm{dist}(T,X_T) + \mathrm{dist}(T,N_T) \leq pq.
$$
\end{lemma}

\begin{proof}
We will say that $T$ is of type 1 if the root and the leaves belong to different sides of the bipartition $P,Q$, and is of type 2 otherwise. 
We distinguish the two cases.

First assume that $T$ is of type 1. Let us further assume without loss of generality that the root of $T$ belongs to $Q$ and the leaves to $P$.
The broom $T$ therefore contains $2t-1$ nonempty sequences of vertices $\alpha_1,\beta_1,\alpha_2,\beta_2,\ldots,\alpha_t$, for some $t > 0$, starting from the root to the leaves,  where $\alpha_i \subseteq Q$ and $\beta_i \subseteq P$, followed by a set $\beta_t\subseteq P$ of leaves. 
For $1 \leq i \leq t$, let $a_i = |\alpha_i|$ and $b_i = |\beta_i|$. An upper bound for $\mathrm{dist}(T,X_T)$ can be computed as follows:
\begin{eqnarray*}
\mathrm{dist}(T,X_T) & \leq & a_1b_1 + (a_1+a_2)b_2 + (a_1+a_2+a_3)b_3 + \ldots + (a_1+a_2+ \ldots + a_t)b_t\\
& = & q b_t + (q-a_t)b_{t-1} + (q-a_t - a_{t-1})b_{t-2} + \ldots + (q-a_t-a_{t-1} - \ldots - a_2)b_1\\
& = & pq - a_2(p-b_t-b_{t-1}-\ldots -b_2) - a_3(p-b_t-b_{t-1}-\ldots -b_3) - \ldots \\
&   & - a_{t-1}(p - b_t - b_{t-1}) - a_t(p-b_t).\\
\end{eqnarray*}

This can be rewritten as $\mathrm{dist}(T,X_T)\leq{pq-r}$ where
$$
r =\sum_{i=2}^ta_i\!\!\left(p-\sum_{j=i}^tb_j\right)\!\!.
$$

Similarly, we can compute an upper bound for $\mathrm{dist}(T,N_T)$ as follows:
\begin{eqnarray*}
\mathrm{dist}(T,N_T) & \leq & a_2b_1 + a_3(b_1+b_2)+a_4(b_1+b_2+b_3)+\ldots \\
&   & + a_{t-1}(b_1+b_2+\ldots + b_{t-2}) + a_t(b_1+b_2+\ldots + b_{t-1})\\
& = & a_t(p-b_t) + a_{t-1}(p-b_t-b_{t-1}) + \ldots + a_3(p-b_t-b_{t-1}-\ldots - b_3) + \\
&   & a_2(p-b_t-b_{t-1}-\ldots - b_2).\\
\end{eqnarray*}

As, the last expression is precisely $r$, we obtain $\mathrm{dist}(T,N_T)\leq{r}$ and therefore
$$
\mathrm{dist}(T,X_T) + \mathrm{dist}(T,N_T) \leq pq,
$$
as desired. Now assume that $T$ is of type 2. In that case, we further require without loss of generality that the handle of $T$ consists of $2t-2$ nonempty sequences of vertices $\alpha_1,\beta_1,\alpha_2,\beta_2,\ldots,\beta_{t-1}$, for some $t > 0$, where $\alpha_i \subseteq Q$ and $\beta_i \subseteq P$. 
The set of the leaves of $T$ is denoted by $\alpha_t\subseteq Q$.
For $1 \leq i \leq t$, let $a_i = |\alpha_i|$ and $b_i = |\beta_i|$.

An upper bound for $\mathrm{dist}(T,X_T)$ can be computed as follows:
\begin{eqnarray*}
\mathrm{dist}(T,X_T) & \leq & a_1b_1 + (a_1+a_2)b_2 + \ldots + (a_1+a_2+ \ldots + a_{t-2})b_{t-2} + (a_1+a_2+\ldots + a_{t-1})b_{t-1}\\
& = & a_1p + a_2(p-b_1) + \ldots + a_{t-2}(p - b_1-b_2- \ldots - b_{t-3}) + a_{t-1}(p - b_1 -b_2 - \ldots - b_{t-2})\\
& = & pq - a_tp - a_{t-1}(b_1+b_2+\ldots + b_{t-2}) - a_{t-2}(b_1+b_2+\ldots + b_{t-3}) - \ldots\\
&   & - a_3(b_1+b_2) - a_2b_1\\
& = & pq - a_t(b_1+b_2+\ldots + b_{t-1}) - a_{t-1}(b_1+b_2+\ldots + b_{t-2}) - \ldots - a_3(b_1+b_2) - a_2b_1.\\
\end{eqnarray*}

Again, this can be rewritten as $\mathrm{dist}(T,X_T)\leq{pq-r}$ where
$$
r=\sum_{i=2}^ta_i\sum_{j=1}^{t-1}b_j.
$$

Finally, we can compute an upper bound for $\mathrm{dist}(T,N_T)$ as follows:
$$
\mathrm{dist}(T,N_T) \leq a_2b_1 + a_3(b_1+b_2) + \ldots + a_{t-1}(b_1+b_2+\ldots + b_{t-2}) + a_t(b_1+b_2+\ldots + b_{t-1}).
$$

The right hand side of this inequality is precisely $r$ and therefore, $\mathrm{dist}(T,X_T) + \mathrm{dist}(T,N_T) \leq pq$.
\end{proof}

We are now ready to prove an upper bound on the distance of two brooms on $K_{p,q}$.

\begin{lemma}
\label{l3sp}
Let $p$ and $q$ be positive integers. For any two brooms $T_1$ and $T_2$ on $\mathrm{K}_{p,q}$,
$$
\mathrm{dist}(T_1,T_2) \leq 2pq.
$$
\end{lemma}
\begin{proof}
We will consider two paths from $T_1$ to $T_2$ and choose the shortest. The first path is of the form
$$
P_1 = T_1 \to X_{T_1} \to N_{T_2} \to T_2,
$$
where the arrow in $A\to B$ denotes a shortest path between $A$ and $B$, and the second path is of the form
$$
P_2 = T_1 \to N_{T_1} \to X_{T_2} \to T_2.
$$

Clearly, $\mathrm{dist}(X_{T_1},N_{T_2}) = \mathrm{dist}(X_{T_2},N_{T_1}) = pq$. We assume without loss of generality that $P_1$ is shorter than $P_2$. We know from Lemma \ref{lemme-dist-sp} that there are two integers $r_1,r_2 \geq 0$ such that $\mathrm{dist}(T_1,N_{T_1}) \leq pq - r_1$ and $\mathrm{dist}(T_1,X_{T_1}) \leq r_1$ (or vice-versa) and $\mathrm{dist}(T_2,N_{T_2}) \leq pq - r_2$ and $\mathrm{dist}(T_2,X_{T_2}) \leq r_2$ (or vice-versa). Therefore, we only need to consider the following four cases.
\begin{itemize}
\item[(a)] Suppose that $\mathrm{dist}(T_1,X_{T_1}) \leq pq - r_1$ and $\mathrm{dist}(N_{T_2},T_2) \leq pq - r_2$. As $P_1$ is shorter than $P_2$,
$$
pq - r_1 + pq - r_2 \leq r_1 + r_2
$$
which implies that $r_1+r_2 \geq pq$. Therefore, $\mathrm{dist}(T_1,X_{T_1}) + \mathrm{dist}(N_{T_2},T_2) \leq 2pq - (r_1+r_2) \leq pq$.

\item[(b)] Suppose that $\mathrm{dist}(T_1,X_{T_1}) \leq pq - r_1$ and $\mathrm{dist}(N_{T_2},T_2) \leq r_2$. As $P_1$ is shorter than $P_2$,
$$
pq - r_1 + r_2 \leq r_1 + pq - r_2.
$$

Therefore, $r_1 \geq r_2$ and in turn, $\mathrm{dist}(T_1,X_{T_1}) + \mathrm{dist}(N_{T_2},T_2) \leq pq - r_1+r_2 \leq pq$.

\item[(c)] Suppose that $\mathrm{dist}(T_1,X_{T_1}) \leq r_1$ and $\mathrm{dist}(N_{T_2},T_2) \leq pq - r_2$. As $P_1$ is shorter than $P_2$,
$$
r_1 + pq - r_2 \leq pq - r_1 + r_2
$$
and as a consequence, $r_2 \geq r_1$. Hence, $\mathrm{dist}(T_1,X_{T_1}) + \mathrm{dist}(N_{T_2},T_2) \leq r_1+ pq - r_2 \leq pq$.

\item[(d)] Finally, suppose that $\mathrm{dist}(T_1,X_{T_1}) \leq r_1$ and $\mathrm{dist}(N_{T_2},T_2) \leq r_2$. As $P_1$ is shorter than $P_2$,
$$
r_1 + r_2 \leq pq - r_1 + pq - r_2
$$

It follows that $r_1 + r_2 \leq pq$ and therefore that $\mathrm{dist}(T_1,X_{T_1}) + \mathrm{dist}(N_{T_2},T_2) \leq r_1+ r_2 \leq pq$.
\end{itemize}
Since $\mathrm{dist}(X_{T_1},N_{T_2}) = pq$, the length of path $P_1$ is at most $2pq$, as desired. 
\end{proof}

Theorem \ref{teo1-sp} is a direct consequence of Lemmas \ref{l2sp} and \ref{l3sp}.

\section{Discussion}

In a recent preprint, Berendsohn give tight bounds on the diameter of associahedra of \textit{caterpillars}~\cite{B21}. 
In particular, he shows that when the graph $G$ is a caterpillar with exactly one leaf attached to each vertex of the spine,
then the diameter of $\mathcal A(G)$ is $\Omega(n\log n)$. This improves our Theorem~\ref{thm:pw2tight}, in the sense that the
result already holds for graphs of pathwidth one! (Indeed, the graphs of pathwidth one are exactly the caterpillars.)

Our work raises several questions. 
\begin{itemize}
\item The bound given in Theorem~\ref{thm:twb} should be tightened.
\item We proved that the diameter of associahedra of trivially perfect, complete split, and complete bipartite graphs is always
at most twice the number of edges of the graph.
All these graphs are \textit{cographs}: graphs without paths on four vertices as induced subgraphs. 
We propose the following question:
\begin{question}
Is it true that if $G$ is a connected cograph on $m$ edges, then $\delta(\mathcal{A}(G)) \leq 2m$?
\end{question}
\item We determined the diameter of the associahedra of complete bipartite graphs in the unbalanced case, when one of the two parts is larger than the other. While we also prove a general upper bound, the diameter of balanced bipartite graphs is still unknown. 
\begin{question}
What is the exact value of $\delta(\mathcal{A}(\mathrm{K}_{p,q}))$ when $\displaystyle\frac{p}{4}\leq{q}\leq4p$?
\end{question}

\end{itemize}

\noindent\textbf{Acknowledgement.} This work was partially supported by the French-Belgian PHC Project number 42703TD.

\bibliographystyle{plain}
\bibliography{GraphAssociahedra}

\end{document}